\documentclass{amsproc}
\usepackage{amssymb}
\usepackage{amsmath}
\usepackage{enumitem}
\usepackage{mathrsfs}
\usepackage{colonequals}
\usepackage[all]{xy}
\usepackage{comment}
\usepackage{marginnote}
\usepackage[english]{babel}

\usepackage{hyperref}

\hypersetup{pdftitle={The Belyi degree of a curve is computable},pdfauthor={Ariyan Javanpeykar, John Voight}}
\hypersetup{colorlinks=true,linkcolor=blue,anchorcolor=blue,citecolor=blue}

\RequirePackage{mathrsfs} 

\numberwithin{equation}{section} 
\numberwithin{figure}{section}

\newtheorem{theorem}[equation]{Theorem}
\newtheorem{lemma}[equation]{Lemma}
\newtheorem{proposition}[equation]{Proposition}
\newtheorem{corollary}[equation]{Corollary}

\theoremstyle{definition}
\newtheorem{definition}[equation]{Definition}

\theoremstyle{remark}
\newtheorem{remark}[equation]{Remark}
\newtheorem{example}[equation]{Example}

\newenvironment{enumroman}
{\begin{enumerate}}
{\end{enumerate}}

\DeclareMathOperator{\Beldeg}{Beldeg} 
\newcommand{\Bdeg}{\Beldeg}

\DeclareMathOperator{\Gal}{Gal} 

\DeclareMathOperator{\Aut}{Aut}

\DeclareMathOperator{\Spec}{Spec}

\DeclareMathOperator{\PSL}{PSL}

\DeclareMathOperator{\ord}{ord}

\DeclareMathOperator{\GL}{GL}

\DeclareMathOperator{\PGL}{PGL}

\DeclareMathOperator{\opdiv}{div}
\DeclareMathOperator{\Isom}{Isom}

\newcommand{\Qbar}{\overline{\mathbb{Q}}}

\newcommand{\Q}{\QQ}
\newcommand{\QQbar}{\Qbar}

\newcommand{\scrL}{\mathscr{L}}

\newcommand\FF{\mathbb{F}}
\newcommand\PP{\mathbb{P}}

\newcommand\ZZ{\mathbb{Z}}

\newcommand\QQ{\mathbb{Q}}
\newcommand\RR{\mathbb{R}}
\newcommand\CC{\mathbb{C}}

\newcommand\Z{\mathbb{Z}}

\newcommand\scrO{\mathscr{O}}
\newcommand\Hzero{\mathrm{H}^0}

\newcommand{\Htwost}{\textup{\textbf{\textsf{H}}}^{2*}}

\newcommand{\p}{\mathbb{P}}

\DeclareMathOperator{\Bel}{Bel}

\renewcommand{\leq}{\leqslant}

\renewcommand{\geq}{\geqslant}

\newcommand{\Belyi}{Bely\u{\i}}

\newcommand{\defi}[1]{\textsf{#1}} 	

\title{The Belyi degree of a curve is computable}

\author{Ariyan Javanpeykar}
\address{Mathematical Institute \\ Johannes-Gutenberg University\\
Mainz, Germany}
\email{peykar@uni-mainz.de}

\author{John Voight}
\address{Department of Mathematics,
  Dartmouth College, 6188 Kemeny Hall, Hanover, NH 03755, USA}
\email{jvoight@gmail.com}

\begin{document}

\begin{abstract}
We exhibit an algorithm that, given input a curve $X$ over a number field, computes as output the minimal degree of a \Belyi\ map $X \to \PP^1$.  We discuss in detail the example of the Fermat curve of degree $4$ and genus $3$.
\end{abstract}
 
\maketitle

\section{Introduction}
Let $\Qbar \subset \CC$ be the algebraic closure of $\QQ$ in $\CC$.  Let $X$ be a smooth projective connected curve over $\Qbar$; we call $X$ just a \defi{curve}.  \Belyi\ proved \cite{Belyi1,Belyi2} that there exists a finite morphism $\phi\colon X\to \p^1_{\Qbar}$ unramified away from $\{0,1,\infty\}$; we call such a map $\phi$ a \defi{\Belyi\ map}.  

Grothendieck applied \Belyi's theorem to show that the action of the absolute Galois group of $\mathbb Q$ on the set of   dessins d'enfants is faithful \cite[Theorem 4.7.7]{Szamuely}.  This observation began a flurry of activity  \cite{Schneps}: for instance, the theory of dessins d'enfants was used to show that  the action of the Galois group of $\mathbb Q$ on the set of connected components of the coarse moduli space of surfaces of general type  is  faithful   \cite{BCG, Torres}.  Indeed, the applications of \Belyi's theorem are vast.

In this paper, we consider \Belyi\ maps from the point of view of algorithmic number theory.  We define the \defi{\Belyi\ degree} of $X$, denoted by $\Beldeg(X) \in \Z_{\geq 1}$, to be the minimal degree of a \Belyi\ map $X\to \p^1_{\Qbar}$.  
This integer appears naturally    in Arakelov theory,    the study of rational points on curves, and computational aspects of algebraic curves \cite{BiSt, J, JvK, SijslingVoight}.   It was defined and studied first by 
Li\c{t}canu \cite{Litcanu}, whose work suggested that the \Belyi\ degree behaves like a height.

The aim of this paper is to show that the \Belyi\ degree is an effectively computable invariant of the curve $X$.

\begin{theorem} \label{mainresult}
There exists an algorithm that, given as input a curve $X$ over $\Qbar$, computes as output the \Belyi\ degree $\Beldeg(X)$.
\end{theorem}

The input curve $X$ is specified by equations in projective space with coefficients in a number field.  In fact, the resulting equations need only provide a birational model for $X$, as one can then effectively compute a smooth projective model birational to the given one.

In the proof of his theorem, \Belyi\ provided an algorithm that, given as input a finite set of points $B\subset \p^1(\Qbar)$, computes a \Belyi\ map $\phi\colon \p^1_\Q\to \p^1_\Q$ (defined over $\Q$) such that $\phi(B)\subseteq \{0,1,\infty\}$.  Taking $B$ to be the ramification set of any finite map $X \to \p^1_{\Qbar}$, it follows that there is an algorithm that, given as input a curve $X$ over $\Qbar$, computes as output an \emph{upper bound} for $\Beldeg(X)$.  Khadjavi \cite{Khadjavi} has given an explicit such upper bound---see Proposition \ref{prop: khad} for a precise statement.  So at least one knows that the \Belyi\ degree has a computable upper bound.  However, neither of these results give a way to compute the \Belyi\ degree: what one needs is the ability to test if a curve $X$ has a \Belyi\ map of a given degree $d$.  Exhibiting such a test is the content of this paper, as follows.

A \defi{partition triple} of $d$ is a triple of partitions $\lambda=(\lambda_0,\lambda_1,\lambda_\infty)$ of $d$.  The ramification type associates to each isomorphism class of \Belyi\ map of degree $d$ a partition triple $\lambda$ of $d$.  

\begin{theorem} \label{mainresult:really}  
There exists an algorithm that, given as input a curve $X$ over $\Qbar$, an integer $d\geq 1$ and a partition triple $\lambda$ of $d$, determines if there exists a \Belyi\ map $\phi\colon X \to \PP^1_{\Qbar}$ of degree $d$ with ramification type $\lambda$; and, if so, gives as output a model for such a map $\phi$.  
\end{theorem}

Theorem \ref{mainresult:really} implies Theorem \ref{mainresult}: for each $d \geq 1$, we loop over partition triples $\lambda$ of $d$ and we call the algorithm in Theorem \ref{mainresult:really}; we terminate and return $d$ when we find a map.

The plan of this paper is as follows. In section \ref{section: Belyi degree}, we begin to study the \Belyi\ degree and gather some of its basic properties. For instance, we observe that, for all odd $d\geq 1$, there is a curve of \Belyi\ degree $d$. We also recall Khadjavi's effective version of Belyi's theorem.  In section \ref{section: proof}, we prove Theorem \ref{mainresult:really} by exhibiting equations for the space of Belyi maps on a curve with given degree and ramification type: see Proposition \ref{prop:eqns}.  These equations can be computed in practice, but unfortunately in general it may not be practical to detect if they have a solution over $\Qbar$.  In section \ref{sec:secondproof}, we sketch a second proof, which is much less practical but still proves the main result.  Finally, in section \ref{section:fermat} we discuss in detail the example of the Fermat curve $x^4+y^4=z^4$ of genus $3$.

The theory of \Belyi\ maps in characteristic $p>0$ is quite different,
and our main results rely fundamentally on the structure of the
fundamental group of $\mathbb{C}\setminus\{0,1\}$, so we work over
$\Qbar$ throughout.  However, certain intermediate results, including
Lemma \ref{lemma:alg_for_isoms0}, hold over a general field.

\subsection*{Acknowledgements} This note grew out of questions asked to the authors by Yuri Bilu, Javier Fres\'an,  David Holmes, and Jaap Top, and the authors are grateful for these comments.  The authors also wish to thank Jacob Bond, Michael Musty, Sam Schiavone, and the anonymous referee for their feedback.  Javanpeykar gratefully acknowledges support from SFB Transregio/45.  Voight was supported by an NSF CAREER Award (DMS-1151047) and a Simons Collaboration Grant (550029).

\section{The Belyi degree}\label{section: Belyi degree}

In this section, we collect basic properties of the \Belyi\ degree.  Throughout, a \defi{curve} $X$ is a smooth projective connected variety of dimension $1$ over $\Qbar$; we denote its genus by $g=g(X)$. We write $\p^n$ and $\mathbb A^n$ for the schemes $\p^n_{\Qbar}$ and $\mathbb A^n_{\Qbar}$, respectively.  A \defi{\Belyi\ map} on $X$ is a finite morphism $X \to \p^1$ unramified away from $\{0,1,\infty\}$.  Two \Belyi\ maps $\phi\colon X \to \p^1$ and $\phi'\colon X' \to \p^1$ are \defi{isomorphic} if there exists an isomorphism $i\colon X \xrightarrow{\sim} X'$ such that $\phi' \circ i = \phi$.
For $d \geq 1$, define $\Bel_d(X)$ to be the set of isomorphism classes of \Belyi\ maps of degree $d$ on $X$, and let $\Bel(X) \colonequals \bigcup_d \Bel_d(X)$.  

\begin{definition}\label{def:belyi_deg}
The \defi{\Belyi\ degree} of $X$, denoted $\Beldeg(X) \in \Z_{\geq 1}$, is the minimal degree of a \Belyi\ map on $X$.
\end{definition}

In our notation, the \Belyi\ degree of $X$ is the smallest positive integer $d$ such that $\Bel_d(X)$ is non-empty.

\begin{lemma} \label{lem:fin}
Let $C\in \RR_{\geq 1}$. Then
  the set of isomorphism classes of curves $X$ with $\Beldeg(X)\leq C$ is finite.
\end{lemma}

For an upper bound on the number of isomorphism classes of curves $X$ with $\Beldeg(X)\leq C$ we refer to Li\c{t}canu \cite[Th\'eor\`eme 2.1]{Litcanu}. 

\begin{proof}
The monodromy representation provides a bijection between isomorphism classes of \Belyi\ maps of degree $d$ and permutation triples from $S_d$ up to simultaneous conjugation; and there are only finitely many of the latter for each $d$.  Said another way: the (topological) fundamental group of $\p^1(\CC) \smallsetminus \{0,1,\infty\}$ is finitely generated, and so there are only finitely many conjugacy classes of subgroups of bounded index.
\end{proof}

\begin{remark}
One may also restrict to $X$ over a number field $K \subseteq \Qbar$ and ask for the minimal degree of a \Belyi\ map defined over $K$: see Zapponi \cite{Zapponi} for a discussion of this notion of relative \Belyi\ degree.
\end{remark}

Classical modular curves have their \Belyi\ degree bounded above by the index of the corresponding modular group, as follows.

\begin{example}
Let $\Gamma \leq \PSL_2(\ZZ)$ be a finite index subgroup, and let $X(\Gamma) \colonequals \Gamma \backslash \Htwost$ where $\Htwost$ denotes the completed upper half-plane.  Then $\Bdeg(X(\Gamma)) \leq [\PSL_2(\ZZ):\Gamma]$, because the natural map $X(\Gamma) \to X(1) = \PSL_2(\ZZ) \backslash \Htwost \xrightarrow{\sim} \PP^1_{\CC}$ descends to $\QQbar$ and defines a \Belyi\ map, where the latter isomorphism is the normalized modular $j$-invariant $j/1728$.
\end{example}

A lower bound on the \Belyi\ degree may be given in terms of the genus, as we show now.

\begin{proposition}\label{prop:rh}
For every curve $X$, the inequality $\Beldeg(X) \geq 2g(X)+1$ holds.
\end{proposition}
\begin{proof}
By the Riemann--Hurwitz theorem, the degree of a map is minimized when its ramification is total, so for a \Belyi\ map of degree $d$ on $X$ we have
\[ 2g-2 \leq -2d+3(d-1) = d-3, \]
and therefore $d \geq 2g+1$.
\end{proof}

As an application of Proposition \ref{prop:rh}, we now show that gonal maps on curves of positive genus are not \Belyi\ maps.

\begin{corollary} Let $X$ be a curve of gonality $\gamma$.
A finite map $\phi\colon X \to \PP^1$ with $\deg \phi = \gamma$ is a \Belyi\ map only if $\phi$ is an isomorphism.
\end{corollary}

\begin{proof}
If $g(X)=0$, then the result is clear.  On the other hand, the gonality of $X$ is bounded above by $\lceil g(X)/2 \rceil+1$ by Brill--Noether theory \cite[Chapter V]{ACGH1}, and the strict inequality $2g(X)+1 > g(X)/2+1$ holds unless $g(X)=0$, so the result follows from Proposition \ref{prop:rh}.
\end{proof}

\begin{example} \label{example: Belyi is n}
Let $d=2g+1 \geq 1$ be odd, and let $X$ be the curve defined by $y^2-y=x^d$.  Then $X$ has genus $g$, and we verify that the map $y\colon X \to \p^1$ is a \Belyi\ map of degree $d$.  Therefore, the lower bound in Proposition \ref{prop:rh} is sharp for every genus $g$.
\end{example}

\begin{remark} The bound in Proposition \ref{prop:rh} gives a ``topological'' lower bound for the \Belyi\ degree of $X$. One can also give   ``arithmetic'' lower bounds as follows.  Let $p$ be a prime number, and let $X$ be the elliptic curve given by the equation $y^2=x(x-1)(x-p)$ over $\mathbb{Q}$.  Then $X$ has (bad) multiplicative reduction at $p$ and this bad reduction persists over any extension field.    It follows from work of Beckmann \cite{Beckmann} that $\Beldeg(X) \geq p$ (see also Zapponi \cite[Theorem~1.3]{Zapponi}): if $\phi\colon X \to \PP^1$ is a \Belyi\ map of degree $d<p$, then the monodromy group $G$ of $\phi$ has $p \nmid \#G$, and so $\phi$ and therefore $X$ has potentially good reduction at $p$ (in fact, obtained over an extension of $\Q$ unramified at $p$), a contradiction. 
\end{remark}

\begin{example} \label{exm:Fermatz}
 For every $n\geq 1$, the \Belyi\ degree of the Fermat curve 
 \[ X_n\colon x^n+y^n=z^n \subset \PP^2 \] 
 is bounded above by $\Bdeg(X_n) \leq n^2$, because there is a \Belyi\ map 
 \begin{align*}
 X_n &\to \PP^1 \\
 (x:y:z) &\mapsto (x^n:z^n)
 \end{align*} 
 of degree $n^2$.  On the other hand, we have $\Bdeg(X_n) \geq (n-1)(n-2)+1 = n^2 - 3n +3$ by  Proposition~\ref{prop:rh}.
 
For $n=1,2$, we have $X_n \simeq \PP^1$ so $\Bdeg(X_1)=\Bdeg(X_2)=1$.  As observed by Zapponi \cite[Example~1.2]{Zapponi}, for $n=3$, the curve $X_3$ is a genus $1$ curve with $j$-invariant $0$, so isomorphic to $y^2-y=x^3$, and $\Beldeg(X_3)=3$ by Example \ref{example: Belyi is n}.

We consider the case $n=4$ in section~\ref{section:fermat}, and show that $\Bdeg(X_4)=8$ in Proposition \ref{prop:fermat}.
\end{example}



We finish this section with an effective version of \Belyi{}'s theorem, due to Khadjavi \cite{Khadjavi}.  (An effective version was also proven independently by Li\c{t}canu \cite[Th\'eor\`eme 4.3]{Litcanu}, with a weaker bound.)  To give her result, we need the height of a finite subset of $\p^1(\Qbar)$.  For $K$ a number field and $a \in K$, we define the \defi{(exponential) height} to be $H(a) \colonequals \left( \prod_v \max(1,\Vert \alpha\Vert_v) \right)^{1/[K:\Q]}$, where the product runs over the set of absolute values indexed by the places $v$ of $K$ normalized so that the product formula holds \cite[Section 2]{Khadjavi}.  For a finite subset $B \subset \p^1(\Qbar)$, and $K$ a number field over which the points $B$ are defined, we define its \defi{(exponential) height} by $H_B \colonequals \max \{H(\alpha) : \alpha \in B\}$, and we let $N_B$ be the cardinality of the Galois orbit of $B$. 

\begin{proposition} [Effective version of \Belyi{}'s theorem] \label{prop: khad}
Let $B\subset \PP^1(\Qbar)$ be  a finite set. Write $N = N_B$. Then there exists a \Belyi\ map $\phi\colon \p^1\to \p^1$ such that
$\phi(B) \subseteq \{0,1,\infty\}$ and  $$\deg \phi \leq (4NH_B)^{9N^3 2^{N-2}N!}.$$ 
\end{proposition}
\begin{proof}
See Khadjavi \cite[Theorem 1.1.c]{Khadjavi}.
\end{proof}

\begin{corollary}
Let $X$ be a   curve, and let $\pi\colon X\to \p^1$ be a finite morphism with branch locus $B\subset \p^1(\Qbar)$. Write $N= N_B$. Then \[ \Beldeg(X) \leq (4NH_B)^{9N^3 2^{N-2}N!} \deg \pi.\]
\end{corollary}
\begin{proof}
Choose $\phi$ as in Proposition \ref{prop: khad} and consider the composed morphism $\phi\circ \pi$. 
\end{proof}

\section{First proof of Theorem \ref{mainresult:really}} \label{section: proof}

Throughout this section, let $K$ be a number field.  We begin with two preliminary lemmas.

\begin{lemma}\label{lem: configuration}
There exists an algorithm that, given as input an affine variety $X\subset \mathbb A^n$ and $t \geq 1$, computes as output $N \geq 1$ and generators for an ideal $I \subseteq \Qbar[x_1,\ldots,x_N]$ such that the zero locus of $I$ is the variety obtained by removing all the diagonals from $X^t/S_t$.
\end{lemma}

\begin{proof}
Let $X = \Spec \Qbar[x_1,\dots,x_n]/I$. By (classical)  invariant theory (see Sturmfels \cite{Sturmfels}), there is an algorithm to compute the coordinate ring of invariants $\left(\Qbar[x_1,\ldots,x_n]/I\right)^{S_t}$. In other words, there is an algorithm which computes  $$X^t/S_t = \Spec \left(\left(\Qbar[x_1,\ldots,x_n]/I\right)^{S_t}\right).$$  To conclude the proof, note that the complement of a divisor $D=Z(f)$ is again an affine variety, adding a coordinate $z$ satisfying $z f - 1$. 
\end{proof}

\begin{remark}
We will use Lemma \ref{lem: configuration} below to parametrize extra ramification points, write equations in terms of these parameters, and check whether the system of equations has a solution over $\Qbar$.  For this purpose, we need not take the quotient by the symmetric group $S_t$, as the system of equations with unordered parameters has a solution over $\Qbar$ if and only if the one with ordered parameters does.  
\end{remark}

Next, we show how to represent rational functions on $X$ explicitly in terms of a Riemann--Roch basis.  

\begin{lemma}\label{lem:m}
Let $X$ be a curve over $K$ of genus $g$, let $\scrL$ be an ample sheaf on $X$, and let $d$ be a positive integer.  Let
\begin{equation} \label{eqn:deft}
t \colonequals \left\lceil \frac{d+g}{\deg \scrL} \right\rceil. 
\end{equation}
Then, for all $f \in \Qbar(X)$ of degree $d$, there exist $a,b \in \mathrm{H}^0(X_{\Qbar},\scrL^{\otimes t})$ with $b\neq 0$ such that $f=a/b$.
\end{lemma}

\begin{proof}
By definition, we have
\begin{equation}\label{ineq}
t \deg \scrL - d + 1-g \geq 1.
\end{equation} 
Let $\opdiv_\infty f \geq 0$ be the divisor of poles of $f$.  By Riemann--Roch, 
\begin{equation}
\dim_{\Qbar} \Hzero(X_{\Qbar},\scrL^{\otimes t}(-\opdiv_\infty f)) \geq  t\deg \scrL - d + 1-g \geq 1.
\end{equation}
Let 
\[ b \in \Hzero(X_{\Qbar}, \scrL^{\otimes t}(-\opdiv_\infty f)) \subseteq \Hzero(X_{\Qbar},\scrL^{\otimes t}) \] 
be a nonzero element.  Then $fb \in \mathrm H^0(X_{\Qbar},\scrL^{\otimes t})$.  (In effect, we have ``cancelled the poles'' of $f$ by the zeros of $b$, at the expense of possibly introducing new poles supported within $\scrL$.)  Letting $fb=a \in \Hzero(X,\scrL^{\otimes t})$ we have written $f= a/b$ as claimed.
\end{proof}

The quantities in Lemma \ref{lem:m} can be effectively computed, as follows.  Recall that a curve $X$ over $K$ is specified in bits by a set of defining equations in projective space with coefficients in $K$.  (Starting with any birational model for $X$, we can effectively compute a smooth projective model.)

\begin{lemma}\label{lem:m2}
Let $X \subset \PP_K^n$ be a curve over $K$.  Then the following quantities are effectively computable:
\begin{enumroman}
\item The genus $g=g(X)$;
\item An effective divisor $D$ on $X$ over $K$ and its degree.
\item Given a divisor $D$ over $K$, a basis for the $K$-vector space $\Hzero(X,\scrO_X(D))$.
\end{enumroman}
\end{lemma}

\begin{proof}
For (a), to compute the genus we compute a Gr\"obner basis for the defining ideal $I$ of $X$, compute its Hilbert polynomial, and recover the (arithmetic equals geometric) genus from the constant term.  For (b), intersecting $X$ with a hyperplane, we obtain an effective divisor $D$ on $X$ over $K$, and its degree is the leading term of the Hilbert polynomial computed in (a).
For (c), it suffices to note that Riemann--Roch calculations can be done effectively: see e.g.\ Coates \cite{Coates} or Hess \cite{Hess}.
\end{proof}



A \defi{ramification type} for a positive integer $d$ is a triple $\lambda = (\lambda_0,\lambda_1,\lambda_\infty)$ of partitions of $d$. For $X$ a curve, $d$ an integer, and $\lambda$ a ramification type, let $\Bel_{d,\lambda}(X) \subseteq \Bel_d(X)$ be the subset of \Belyi\ maps of degree $d$ on $X$ with ramification type $\lambda$.  For the ramification type $\lambda$ and $* \in \{0,1,\infty\}$, let $\lambda_{*,1},\dots,\lambda_{*,r_*}$ be the parts of $\lambda$ (and $r_*$ the number of parts), so
\[ d = \lambda_{*,1} + \dots + \lambda_{*,r_*}. \]
If $\phi\colon X \to \PP^1$ is a \Belyi\ map of degree $d$ with ramification type $\lambda$, then the Riemann--Hurwitz formula is satisfied: 
\begin{equation} \label{eqn:riemmanhurwitzbel}
\begin{aligned}
2g-2 &= -2 d + \sum_{i=1}^{r_0} \left(\lambda_{0,i} - 1\right) + \sum_{i=1}^{r_1} \left(\lambda_{1,i} - 1\right)
+ \sum_{i=1}^{r_\infty} \left(\lambda_{\infty,i} - 1\right) \\
&= d-r_0 - r_1 - r_\infty.
\end{aligned}
\end{equation}

To prove our main theorem, we will show that one can compute equations whose vanishing locus over $\Qbar$ is precisely the set $\Bel_{d,\lambda}(X)$ (see Proposition \ref{prop:eqns}): we call such equations a \defi{model} for $\Bel_{d,\lambda}(X)$.

On our way to prove Proposition \ref{prop:eqns}, we first characterize \Belyi\ maps of degree $d$ with ramification type $\lambda$ among rational functions on a curve written in terms of a Riemann--Roch basis.  This characterization is technical but we will soon see that it is quite suitable for our algorithmic application.
 
\begin{proposition} \label{prop:claim}   Let $d\geq 1$ be an integer and let $\lambda=(\lambda_0,\lambda_1,\lambda_\infty)$ be a ramification type for $d$ with $r_0,r_1,r_\infty$ parts, respectively.  Let $X$ be a curve  over  a number field $K$.  Let $g$ be the genus of $X$, and suppose that 
\begin{equation} \label{eqn:2g2dr}
2g-2=d-r_0-r_1-r_\infty.
\end{equation}

Let $D_0$ be an effective divisor of degree $d_0$, and let $\scrL = \scrO_X(D_0)$.  Let $t\geq 1$ be the smallest positive integer such that $t \deg \scrL - d+1-g \geq 1$. Let $g_1,\ldots,g_n \in K(X)$  be a basis for the $K$-vector space  $\mathrm{H}^0(X,\scrL^{\otimes t})$. 
Let $0\leq k, l \leq n$ be integers. Let $m$ be minimal so that $g_k, g_l\in \mathrm{H}^0(X_{\overline{\mathbb{Q}}}, \scrL^{\otimes m}) \subseteq \mathrm{H}^0(X_{\overline{\mathbb{Q}}},\scrL^{\otimes t})$.
Let $$\phi:= \frac{a}{b} =  \frac{a_1 g_1+\ldots +a_k g_k}{ b_1 g_1 +\ldots + b_l g_l } $$ be a nonconstant rational function with $a_1,\dots,b_l \in \Qbar$.

Then the rational function $\phi$ lies in $\Bel_{d,\lambda}(X)(\Qbar)$ if and only if there exists a partition $\mu=\mu_1+\dots+\mu_s$ of $md_0 -d$,  distinct points
\[ P_{1},\dots,P_{r_0},\,Q_{1},\dots,Q_{r_1},\,R_{1},\dots,R_{r_\infty} \in X(\Qbar) \]
and distinct points 
\[ Y_1,\dots,Y_s \in X(\Qbar), \]
allowing these two sets of points to meet, such that
\begin{equation} \label{bel}
\begin{aligned}
\opdiv(a) & = \sum_{i=1}^{r_0} \lambda_{0,i}[P_i] + \sum_{i=1}^{s} \mu_i [Y_i] - mD_0 \\
\opdiv(a-b) & = \sum_{i=1}^{r_1} \lambda_{1,i}[Q_i] + \sum_{i=1}^{s} \mu_i [Y_i] - mD_0 \\
\opdiv(b) & = \sum_{i=1}^{r_\infty} \lambda_{\infty,i}[R_i] + \sum_{i=1}^{s} \mu_i [Y_i] - mD_0.
\end{aligned}
\end{equation}
\end{proposition}

\begin{proof} 
We first prove the implication $(\Leftarrow)$ of the proposition.  Suppose $\phi$ satisfies the equations \eqref{bel}.  Then 
\[ \opdiv \phi = \opdiv(a) - \opdiv(b) = \sum_{i=1}^{r_0} \lambda_{0,i}[P_i] - \sum_{i=1}^{r_\infty} \lambda_{\infty,i}[R_i]; \]
since the set of points $\{P_1,\dots,P_{r_0}\}$ is disjoint from $\{R_1,\dots,R_{r_\infty}\}$, we have $\deg \phi = d$.  We see some ramification in $\phi\colon X \to \PP^1$ above the points $0,1,\infty$ according to the ramification type $\lambda$, specified by the equations \eqref{bel}; let $\rho$ be the degree of the remaining ramification locus.  We claim there can be no further ramification.  Indeed, the Riemann--Hurwitz formula gives
\begin{equation}
\begin{aligned}
2g-2 &= -2 d + \sum_{i=0}^{r_0} \left(\lambda_{0,i} - 1\right) + \sum_{i=0}^{r_1} \left(\lambda_{1,i} - 1\right)
+ \sum_{i=0}^{r_\infty} \left(\lambda_{\infty,i} - 1\right) + \rho \\
&= d-r_0-r_1-r_\infty+\rho.
\end{aligned}
\end{equation}
On the other hand, we are given the equality \ref{eqn:2g2dr}, so $\rho=0$.  Therefore $\phi \in \Bel_{d,\lambda}(\Qbar)$.  

We now prove the other implication $(\Rightarrow)$.  Suppose $\phi \in \Bel_{d,\lambda}(\Qbar)$.  We have 
\begin{equation} \label{eqn:fabdiv}
\opdiv(\phi) = \opdiv(a)-\opdiv(b) = \sum_{i=1}^{r_0} \lambda_{0,i}[P_i] - \sum_{i=1}^{r_\infty} \lambda_{\infty,i}[R_i] 
\end{equation}
and
\begin{equation} 
\opdiv(\phi-1) = \opdiv(a-b)-\opdiv(b) = \sum_{i=1}^{r_1} \lambda_{1,i}[Q_i] - \sum_{i=1}^{r_\infty} \lambda_{\infty,i}[R_i] 
\end{equation}
for distinct points $P_1,\dots,P_{r_0},Q_1,\dots,Q_{r_1},R_1,\dots,R_{r_\infty} \in X(\Qbar)$.  Moreover, since $a \in \Hzero(X,\scrL^{\otimes t})$, we have
\[ \opdiv(a) = \sum_{i=1}^{r_0} \lambda_{0,i}[P_i] + E - mD_0 \]
for some effective divisor $E$ (not necessarily disjoint from $D_0$) with $\deg E = md_0-d$; from \eqref{eqn:fabdiv} we obtain
\[ \opdiv(b) = \sum_{i=1}^{r_\infty} \lambda_{\infty,i}[R_i] + E - mD_0. \]
Writing out $E=\sum_{i=1}^s \mu_i [Y_i]$ with $Y_i$ distinct as an effective divisor and arguing similarly for $\opdiv(a-b)$, we conclude that the equations \eqref{bel} hold.
\end{proof}

\begin{remark}
$\Bel_{d,\lambda}(X)$ is a (non-positive dimensional) Hurwitz space: see for instance Bertin--Romagny \cite[Section~6.6]{BertinRomagny} (but also Mochizuki \cite{Mochizuki} and Romagny--Wewers \cite{RomagnyWewers}).   Indeed, for a scheme $S$ over $\Qbar$, let $\underline{\Bel}_{d,\lambda,X}(S)$ be the groupoid whose objects are tuples $(\phi\colon Y\to \mathbb{P}^1_S,\,g\colon Y\to X_S)$, where $Y$ is a smooth proper geometrically connected curve over $S$, the map $\phi\colon Y\to \mathbb{P}^1_S$ is a finite flat finitely-presented morphism of degree $d$ ramified  only over $0,1,\infty$ with ramification type $\lambda$, and $g$ is an isomorphism of $S$-schemes. This defines a (possibly empty) separated 
finite type Deligne--Mumford  
algebraic stack $\underline{\Bel}_{d,\lambda,X}$ over $\Qbar$ which is usually referred to as a \emph{Hurwitz stack}.   Its coarse space, denoted by $\Bel_{d,\lambda,X}$, is usually referred to as a \emph{Hurwitz space}.  Since the set of $\Qbar$-points $\Bel_{d,\lambda,X}(\Qbar)$ of its coarse space  $\Bel_{d,\lambda,X}$ is naturally in bijection with $\Bel_{d,\lambda}(X)$, one could say that the following proposition says that there is an algorithm to compute a model for the Hurwitz space $\Bel_{d,\lambda,X}$.
\end{remark}

We now prove the following key ingredient to our main result.

\begin{proposition}\label{prop:eqns}
There exists an algorithm that, given as input a curve $X$ over $\Qbar$, an integer $d$, and a ramification type $\lambda$ of $d$, computes a model for $\Bel_{d,\lambda}(X)$.
\end{proposition}

\begin{proof}
Let $K$ be a field of definition of $X$ (containing the coefficients of the input model).  Applying the algorithm in Lemma \ref{lem:m2} to  $X$ over $K$,
we compute the genus $g$ of $X$.

Recall the Riemann--Hurwitz formula \eqref{eqn:riemmanhurwitzbel} for a \Belyi\ map.  If the Riemann--Hurwitz formula is not satisfied for $d$ and the ramification type $\lambda$, there is no \Belyi\ map of degree $d$ with ramification type $\lambda$ on $X$ (indeed, on any curve of genus $g$), and the algorithm gives trivial output.  So we may suppose that \eqref{eqn:riemmanhurwitzbel} holds.

Next, we compute an effective divisor $D_0$ on $X$ with $\scrL := \scrO_X(D_0)$ and its degree $d_0 := \deg D_0$.  Let   
\[ t \colonequals \left\lceil \frac{d+g}{\deg \scrL} \right\rceil \]
as in \eqref{eqn:deft}.  By Lemma \ref{lem:m2}, we may compute a $K$-basis  $g_1, \ldots, g_n$ of $\mathrm{H}^0(X,\scrL^{\otimes t})$.  Then by Lemma \ref{lem:m}, if $\phi \in \Qbar(X)$ is a degree $d$ rational function on $X$, then there exist $a_1,\dots,a_n,b_1,\dots,b_n \in \Qbar$ such that $a =\sum_{i=1}^n a_i g_i$ and $b= \sum_{i=1}^n b_i g_i$ satisfy $\phi=a/b$. 

 We now give algebraic conditions on the coefficients $a_i,b_j$ that characterize the subset $\Bel_{d,\lambda}(X)$.  There is a rescaling redundancy in the ratio $a/b$ so we work affinely as follows.  We loop over pairs $0 \leq k,\ell \leq n$ and consider functions 
\begin{equation} \label{eqn:fa0b0}
\phi = \frac{a}{b}=\frac{a_1g_1 + \dots + a_{k-1} g_{k-1} + a_kg_k}{b_1g_1 + \dots + b_{\ell-1} g_{\ell-1} + g_\ell} 
\end{equation}
with $a_k \neq 0$.  Every function $\phi=a/b$ arises for a unique such $k,\ell$.   
Let $m$ be minimal so that $g_k,g_\ell \in \Hzero(X_{\Qbar},\scrL^{\otimes m}) \subseteq \Hzero(X_{\Qbar},\scrL^{\otimes t})$. 
 
Note that Proposition \ref{prop:claim} characterizes precisely when a rational function of the form (\ref{eqn:fa0b0}) lies in $\Bel_{d,\lambda}(X)(\Qbar)$.
Thus, by Proposition \ref{prop:claim}, we may finish by noting that the equations \eqref{bel} can be written explicitly.    To this end, we loop over the partitions $\mu$ and consider the configuration space of $r_0+r_1+r_\infty$ and $s$ distinct points (but allowing the two sets to meet), which can be effectively computed by Lemma \ref{lem: configuration}. 
Next, we write $D_0=\sum_i \rho_i [D_{0i}]$ and loop over the possible cases where one of the points $P_i,Q_i,R_i,Y_i$ is equal to one of the points $D_{0i}$ or they are all distinct from $D_{0i}$.  In each case, cancelling terms when they coincide, we impose the vanishing conditions on $a,a-b,b$ with multiple order vanishing defined by higher derivatives, in the usual way.  For each such function, we have imposed that the divisor of zeros is at least as large in degree as the function itself, so there can be no further zeros, and therefore the equations \eqref{bel} hold for any solution to this large system of equations.  
\end{proof}

\begin{example} 
We specialize Proposition \ref{prop:eqns} to the case $X = \PP^1$. 
(The \Belyi\ degree of $\PP^1$ is $1$, but it is still instructive to see what the equations \eqref{bel} look like in this case.)  Let $X=\PP^1$ with coordinate $x$, defined by $\ord_\infty x = -1$.  We take $D_0=(\infty)$.  Then the basis of functions $g_i$ is just $1,\dots,x^d$, and $f=a/b$ is a ratio of two polynomials of degree $\leq d$, at least one of which is degree exactly $d$.  Having hit the degree on the nose, the ``cancelling'' divisor $E=\sum_{i=1}^s \mu_i[Y_i]=0$ in the proof of Proposition \ref{prop:eqns} does not arise, and the equations for $a,b,a-b$ impose the required factorization properties of $f$.  This method is sometimes called the \emph{direct method} and has been frequently used (and adapted) in the computation of \Belyi\ maps using Gr\"obner techniques \cite[\S 2]{SijslingVoight}.
\end{example}

Given equations for the algebraic set $\Bel_{d,\lambda}(X)$, we now prove that there is an algorithm to check whether this set is empty or not.

\begin{lemma}\label{lem:elmn}   There exists an algorithm that, given as input an affine  variety $X$ over $\Qbar$,  computes as output whether $X(\Qbar)$ is empty or not.
\end{lemma}
\begin{proof} Let $I$ be an ideal defining the affine variety $X$ (in some polynomial ring over $\Qbar$).  One can effectively compute a Gr\"obner basis for $I$ \cite[Chapter 15]{Eisenbud}. With a Gr\"obner basis at hand one can easily check whether $1$ is in the ideal or not, and conclude by Hilbert's Nullstellensatz accordingly if $X(\Qbar)$ is empty  or not.  
\end{proof}

\begin{corollary}\label{cor:elmn2}
There exists an algorithm that, given as input a set $S$ with a model computes as output whether $S$ is empty or not.
\end{corollary}
\begin{proof}
Immediate from Lemma \ref{lem:elmn} and the definition of a model for a set $S$ as being given by equations.
\end{proof}


We are now ready to give the first proof of the main result of this note. 

\begin{proof}[First proof of Theorem \ref{mainresult:really}] Let $X$ be a curve over $\Qbar$. Let $d\geq 1$ be an integer, and let $\lambda$ be a ramification type of $d$. To prove the theorem,  it suffices to show that there is an algorithm which computes whether the set $\Bel_{d,\lambda}(X)$ of \Belyi\ maps of degree $d$ with ramification type $\lambda$ is empty. We explain how to use the above results to do this.

By Proposition \ref{prop:eqns}, we may (and do) compute a model for the set $\Bel_{d,\lambda}(X)$. By    Corollary \ref{cor:elmn2}, we can check algorithmically whether this set is empty or not (by using the   model we computed). This means that we can algorithmically check whether $X$ has a \Belyi\ map of degree $d$ with ramification type $\lambda$.
\end{proof}


\section{Second proof of Theorem \ref{mainresult:really}} \label{sec:secondproof}
In this section, we sketch a second proof of Theorem \ref{mainresult:really}. Instead of writing down equations for the Hurwitz space $\Bel_d(X)$, we enumerate all \Belyi\ maps and effectively compute equations to check for isomorphism between curves.  We saw this method already at work in Example \ref{exm:Fermatz}.

Let $X,Y$ be curves over $\Qbar$.  The functor $S \mapsto \Isom_S(X_S,Y_S)$ from the (opposite) category of schemes over $\Qbar$ to the category of sets is representable \cite[Theorem~1.11]{DeligneMumford} by a finite \'etale $\Qbar$-scheme $\underline{\Isom}(X,Y)$. Our next result shows that one can effectively compute a model for the (finite) set $\Isom(X,Y) = \underline{\Isom}(X,Y)(\Qbar)$ of isomorphisms from $X$ to $Y$. Equivalently, one can effectively compute equations for the finite \'etale $\Qbar$-scheme $\underline{\Isom}(X,Y)$.
 
 \begin{lemma}\label{lemma:alg_for_isoms0}
 There exists an algorithm that, given as input curves $X,Y$ over $\Qbar$ with at least one of $X$ or $Y$ of  genus at least $2$, computes a model for the set $\Isom(X,Y)$.
 \end{lemma}

 \begin{proof} 
 We first compute the genera of $X,Y$ (as in the proof of Lemma \ref{lem:m}): if these are not equal, then we correctly return the empty set.  Otherwise, we compute a canonical divisor $K_X$ on $X$ by a Riemann--Roch calculation \cite{Hess} and the image of the pluricanonical map $\varphi\colon X \hookrightarrow \PP^N$ associated to the complete linear series on the very ample divisor $3K_X$ via Gr\"obner bases.  We repeat this with $Y$.  An isomorphism $\Isom(X,Y)$ induces via its action on canonical divisors an element of $\PGL_{N-1}(\Qbar)$ mapping the canonically embedded curve $X$ to $Y$, and vice versa, and so a model is provided by the equations that insist that a linear change of variables in $\PP^N$ maps the ideal of $X$ into the ideal of $Y$, which can again be achieved by Gr\"obner bases.
 \end{proof}

\begin{corollary} \label{cor:alg_for_isoms}
There exists an algorithm that, given as input maps of curves $f\colon X \to \PP^1$ and $h\colon Y \to \PP^1$ over $\Qbar$, computes as output whether there exists an isomorphism $\alpha\colon X \xrightarrow{\sim} Y$ such that $g=\alpha \circ f$ or not.

Similarly, there exists an algorithm that, given as input curves $X,Y$ over $\Qbar$, computes as output whether $X \simeq Y$ or not.
\end{corollary}

As remarked by Ngo--Nguyen--van der Put--Top \cite[Appendix]{Ngoetal}, the existence of an algorithm which decides whether two curves are isomorphic over an algebraically closed field is well-known. We include the following proof for the sake of completeness.

\begin{proof}[Proof of Corollary \ref{cor:alg_for_isoms}]
We compute the genera of $X,Y$ and again if these are different we correctly return as output \emph{no}.  Otherwise, let $g$ be the common genus.  

If $g=0$, we parametrize $X$ and $Y$ to get $X \simeq Y \simeq \PP^1$ and then ask for $\alpha \in \PGL_2(\Qbar)$ to map $f$ to $g$ in a manner analogous to the proof of Lemma \ref{lemma:alg_for_isoms0}.

If $g=1$, we loop over the preimages of $0 \in \PP^1$ in $X$ and $Y$ as origins, we compute Weierstrass equations via Riemann--Roch, and return \emph{no} if the $j$-invariants of $X,Y$ are unequal.  Otherwise, these $j$-invariants are equal and we compute an isomorphism $X \simeq Y$ of Weierstrass equations.  The remaining isomorphisms are twists, and we conclude by checking if there is a twist $\alpha$ of the common Weierstrass equation that maps $f$ to $g$.  

If $g \geq 2$, we call the algorithm in Lemma \ref{lemma:alg_for_isoms0}: we obtain a finite set of isomorphisms, and for each $\alpha \in \Isom(X,Y)$ we check if $h=\alpha \circ f$.

The second statement is proven similarly, ignoring the map.  
\end{proof}

We now give a second proof of our main result.
 
\begin{proof}[Second proof of Theorem \ref{mainresult:really}]
We first loop over integers $d \geq 1$ and all ramification types $\lambda$ of $d$.  For each $\lambda$, we count the number of permutation triples up to simultaneous conjugation with ramification type $\lambda$.  

We then compute the set of \Belyi\ maps of degree $d$ with ramification type $\lambda$ over $\Qbar$ as follows.  There are countably many number fields $K$, and they may be enumerated by a minimal polynomial of a primitive element.  For each number field $K$, there are countably many curves $X$ over $K$ up to isomorphism over $\Qbar$, and this set is computable: for $g=0$ we have only $\PP_K^1$, for $g=1$ we can enumerate $j$-invariants, and for $g \geq 2$ we can enumerate candidate pluricanonical ideals (by Petri's theorem).  Finally, for each curve $X$ over $K$, there are countably many maps $f\colon X \to \PP^1$, and these can be enumerated using Lemma \ref{lem:m}.  Diagonalizing, we can enumerate the entire countable set of such maps.  For each such map $f$, using Gr\"obner bases we can compute the degree and ramification type of $f$, and in particular detect if $f$ is a \Belyi\ map of degree $d$ with ramification type $\lambda$.  Along the way in this (ghastly) enumeration, we can detect if two correctly identified \Belyi\ maps are isomorphic using Corollary \ref{cor:alg_for_isoms}.  Having counted the number of isomorphism classes of such maps, we know when to stop with the complete set of such maps.

Now, to see whether $\Bel_{d,\lambda}(X)$ is nonempty, we just check  using Corollary \ref{cor:alg_for_isoms} whether $X$ is isomorphic to one of the source curves in the set of all \Belyi\ maps of degree $d$ and ramification type $\lambda$.
\end{proof}

\section{The Fermat curve of degree four} \label{section:fermat}

In this section we prove the following proposition, promised in Example \ref{exm:Fermatz}.

\begin{proposition}\label{prop:fermat}
The \Belyi\ degree of the curve $X\colon x^4+y^4 = z^4$ is equal to $8$.
\end{proposition}

\begin{proof}
The curve $X$ is a canonically embedded curve of genus $3$.  By Proposition \ref{prop:rh}, we have $\Bdeg(X) \geq 7$.  On the other hand, $X$ maps to the genus $1$ curve with affine model $z^2=x^4+1$ and $j$-invariant $1728$, and this latter curve has a \Belyi\ map of degree $4$ taking the quotient by its automorphism group of order $4$ as an elliptic curve, equipped with a point at infinity.  Composing the two, we obtain a \Belyi\ map of degree $8$ on $X$ defined by $(x:y:z) \mapsto x^2+z^2$; therefore $\Beldeg(X) \leq 8$. 
So to show $\Beldeg(X)=8$, it suffices to rule out the existence of a \Belyi\ map of degree $7$.  

By enumeration of partitions and the Riemann--Hurwitz formula, we see that the only partition triple of $7$ that gives rise to a \Belyi\ map $\phi\colon X \to \PP^1$ with $X$ of genus $3$ is $(7,7,7)$.  
By enumeration of permutation triples up to simultaneous conjugation, we compute that the \Belyi\ maps of degree $7$ and genus $3$ have three possible monodromy groups: cyclic of order $7$, the simple group $\GL_3(\FF_2) \simeq \PSL_2(\FF_7)$ of order $168$, or the alternating group $A_7$.  We rule these out by consideration of automorphism groups.

As in Lemma \ref{lemma:alg_for_isoms0} but instead using the canonical embedding as $K_X$ is already ample, we have $\Aut(X) \leq \Aut(\PP^2)=\PGL_3(\Qbar)$, and a direct calculation yields that $\Aut(X) \simeq S_3 \rtimes (\ZZ/4\ZZ)^2$ and $\#\Aut(X)=96$.  (For the automorphism group of the general Fermat curve $X_n$ of degree $n \geq 4$, see Leopoldt \cite{Leopoldt} or Tzermias \cite{Tzermias}: they prove that $\Aut(X_n) \simeq S_3 \rtimes (\ZZ/n\ZZ)^2$.)  

The cyclic case is a geometrically Galois map, but $X$ does not have an automorphism of order $7$, impossible.  For the two noncyclic cases, computing the centralizers of the $2+23=25$ permutation triples up to simultaneous conjugation, we conclude that these \Belyi\ maps have no automorphisms.  An automorphism $\alpha \in \Aut(X)$ of order coprime to $7$ cannot commute with a \Belyi\ map of prime degree $7$ because the quotient by $\alpha$ would be an intermediate curve.  So if $X$ had a \Belyi\ map of degree $7$, there would be $96$ nonisomorphic such \Belyi\ maps, but that is too many.  
\end{proof}

\begin{remark}
The above self-contained proof works because of the large automorphism group on the Fermat curve, and it seems difficult to make this strategy work for an arbitrary curve.  
\end{remark}

To illustrate how our algorithms work, we now show how they can be used to give two further proofs of Proposition \ref{prop:fermat}.

\begin{example}
We begin with the first algorithm exhibited in Proposition \ref{prop:eqns}.  We show that $X$ has no \Belyi\ map of degree $7$ with explicit equations to illustrate our method; we finish the proof as above.

We take the divisor $D_0=[D_{01}]$ where $D_{01}=(1:0:1)\in X(\Q)$ and $\deg D_0=d_0=1$.  We write rational functions on $X$ as ratios of polynomials in $\Q[x,y]$, writing $x,y$ instead of $x/z,y/z$.  According to \eqref{ineq}, taking $\scrL=\scrO_X(D_0)$ we need $t-7+1-3 \geq 1$, so we take $t=10$.  By a computation in \textsc{Magma} \cite{Magma}, the space $\Hzero(X,\scrL^{\otimes 10})$ has dimension $n=8$ and basis
\begin{equation}
\begin{aligned}
g_1 &= 1 \\
g_2 &= \frac{x^3+x^2+x+1}{y^3} \\
g_3 &= g_2/y \\
g_4 &= \frac{4(x^3+x^2+x+1) - x^2y^4 - 2xy^4 - 3y^4}{4y^6} \\
g_5 &= g_5/y \\
g_6 &= g_6/y \\
g_7 &= \frac{16(x^3+x^2+x+1)-6x^3y^4 - 10x^2y^4 + xy^8 - 14xy^4 + 3y^8 - 18y^4}{6y^9} \\
g_8 &= \frac{32(x^3+x^2+x+1) - 3x^2y^8 - 8x^2y^4 - 4xy^8 - 16xy^4 - 3y^8 - 24y^4}{32y^{10}}
\end{aligned}
\end{equation}
We compute that $\ord_{D_0} g_i = 0,-3,-4,-6,-7,-8,-9,-10$.

The general case is where $a_8b_8 \neq 0$, for which $k=\ell=8$ and we may take
\[ \phi=\frac{a}{b}=\frac{\sum_{i=1}^{8} a_i g_i}{\sum_{i=1}^{8} b_i g_i} \]
so we let $b_8=1$ and $m=t=10$.  As we already saw in Example \ref{exm:Fermatz}, the only ramification type possible is $\lambda=(7,7,7)$, with $r_0=r_1=r_\infty=1$ and $\lambda_0=\lambda_1=\lambda_\infty=7$.  

We have $md_0-d=10-7=3$, so we consider the partitions of $3$.  We start with the trivial partition $\mu=\mu_1=3$ with $s=1$.  Then the equations \eqref{bel} read, dropping subscripts: we want distinct points $P,Q,R \in X(\Qbar)$ such that $\opdiv(a) \geq 7[P]+3[Y]$ and $\opdiv(a-b) \geq 7[Q]+3[Y]$ and $\opdiv(b) \geq 7[R]+3[Y]$.  

Continuing in the general case, the points $P,Q,R,Y,D_{0}$ are all distinct, each such point belongs to the affine open with $z \neq 0$, and furthermore $x-x(Z)$ is a uniformizer at $Z$ for each point $Z=P,\dots,D_{0}$.  

The conditions for the point $P$ we write as follows: letting $P=(x_{P}:y_{P}:1)$ with unknowns $x_{P},y_{P}$, we add the equation $x_{P}^4+y_{P}^4=1$ so that $P$ lies on the curve $X$, and then (by Taylor expansion) to ensure $\ord_{P} a \geq 7$ we add the equations 
\begin{equation} 
\frac{\partial^j a}{\partial x^j}(x_{P},y_{P}) = \sum_{i=1}^8 a_i \frac{\partial^j g_i}{\partial x^j}(x_{P},y_{P}) = 0
\end{equation}
for $j=0,\dots,6$, and using implicit differentiation on the defining equation of $X$ to obtain $\displaystyle{\frac{\textup{d}y}{\textup{d}x}}=-\frac{x^3}{y^3}$.  For example, the case $j=1$ (asserting that $a$ vanishes to order at least $2$ at $P$, assuming that $a(P)=0$) is 
\begin{equation}
\begin{aligned}
&(3x_{P}^6y_{P}^7 + 3x_{P}^5y_{P}^7 + 3x_{P}^4y_{P}^7 + 3x_{P}^3y_{P}^7 + 3x_{P}^2y_{P}^{13} + 2x_{P}y_{P}^{13} + y_{P}^{13})a_2 \\
    & + (4x_{P}^6y_{P}^6 + 4x_{P}^5y_{P}^6 + 4x_{P}^4y_{P}^6
    + 4x_{P}^3y_{P}^6 + 3x_{P}^2y_{P}^{12} + 2x_{P}y_{P}^{12} + y_{P}^{12})a_3 \\
    &+ 
    (6x_{P}^6y_{P}^4 - \tfrac{11}{2}x_{P}^5y_{P}^8 + 6x_{P}^5y_{P}^4 - 7x_{P}^4y_{P}^8 + 
    6x_{P}^4y_{P}^4  \\
    &\qquad - \tfrac{17}{2}x_{P}^3y_{P}^8 + 6x_{P}^3y_{P}^4 + 3x_{P}^2y_{P}^{10} + 
    4x_{P}y_{P}^{10} + 2y_{P}^{10})a_4 \\
    &+ (7x_{P}^6y_{P}^3 - \tfrac{23}{4}x_{P}^5y_{P}^7 + 
    7x_{P}^5y_{P}^3 - \tfrac{15}{2}x_{P}^4y_{P}^7 + 7x_{P}^4y_{P}^3  \\
    &\qquad - \tfrac{37}{4}x_{P}^3y_{P}^7 + 
    7x_{P}^3y_{P}^3 + 3x_{P}^2y_{P}^9 + 4x_{P}y_{P}^9 + 2y_{P}^9)a_5  \\
    &+ (8x_{P}^6y_{P}^2 - 6x_{P}^5y_{P}^6 + 8x_{P}^5y_{P}^2 - 8x_{P}^4y_{P}^6 + 8x_{P}^4y_{P}^2 \\ 
    &\qquad - 10x_{P}^3y_{P}^6 + 8x_{P}^3y_{P}^2 + 3x_{P}^2y_{P}^8 + 4x_{P}y_{P}^8 + 2y_{P}^8)a_6 \\
    &+
    (5x_{P}^6y_{P}^5 - 24x_{P}^6y_{P} + 11x_{P}^5y_{P}^5 - 24x_{P}^5y_{P} - 
    \tfrac{19}{2}x_{P}^4y_{P}^9 + 17x_{P}^4y_{P}^5 - 24x_{P}^4y_{P} \\
    &\qquad - \tfrac{25}{2}x_{P}^3y_{P}^9 + 
    23x_{P}^3y_{P}^5 - 24x_{P}^3y_{P} + 6x_{P}^2y_{P}^7 + 4x_{P}y_{P}^7 + 3y_{P}^7)a_7 \\
    &+ 
    (10x_{P}^6 - \tfrac{143}{16}x_{P}^5y_{P}^8 - \tfrac{13}{2}x_{P}^5y_{P}^4 + 10x_{P}^5 - 
    \tfrac{37}{4}x_{P}^4y_{P}^8 - 9x_{P}^4y_{P}^4 + 10x_{P}^4 \\
    &\qquad - \tfrac{143}{16}x_{P}^3y_{P}^8 - 
    \tfrac{23}{2}x_{P}^3y_{P}^4 + 10x_{P}^3 + 3x_{P}^2y_{P}^6 + 6x_{P}y_{P}^6 + 3y_{P}^6)a_8 \\
    &=0.
\end{aligned}
\end{equation}
The equations for the points $Q,R$ are the same, with $a-b$ and $b$ in place of $a$, and again for $Y$ but with $a$ \emph{and} $b$ in place of $a$.  We must also impose the conditions that the points are distinct and that $a_8 \neq 0$: for example, to say $P \neq Q$ we introduce the variable $z_{PQ}$ and the equation 
\begin{equation} 
((x_P-x_Q)z_{PQ}-1)((y_P-y_Q)z_{PQ}-1) = 0.
\end{equation}
In this general case, we end up with $8+7+2\cdot 4 + 10 = 33$ variables 
\begin{equation} 
a_1,\dots,a_8,b_1,\dots,b_7,x_{P_1},y_{P_1},x_{Q_1},y_{Q_1},x_{R_1},y_{R_1},x_{Y_1},y_{Y_1},z_{PQ},\dots,z_{RD_0} 
\end{equation}
and $8\cdot 3 + 7 + 10 = 41$ equations.  

Moving on from the general case, we consider also the case where $x$ does not yield a uniformizer for one of the points; that one of the points lies along the line $z=0$; or that some of the points coincide.  After this, we have completed the case $k=\ell=8$, and consider more degenerate cases $(k,\ell)$.  

Finally, we repeat the entire process again with the partitions $\mu=2+1$ and $\mu=1+1+1$.
\end{example}


We conclude by a version of the second proof of our main result, explained in section \ref{sec:secondproof}.

\begin{example}
We compute each \Belyi\ map of degree $7$ and genus $3$ and show that no source curve is isomorphic to $X$.  

As above, there are three cases to consider.  The first cyclic case is the map in Example \ref{example: Belyi is n} above, followed by its post-composition by automorphisms of $\PP^1$ permuting $\{0,1,\infty\}$.  But the curve $y^2-y=x^7$ has an automorphism of order $7$, and $X$ does not.  

The genus $3$ \Belyi\ maps of degree $7$ in the noncyclic case with $2$ permutation triples up to conjugation was computed by Klug--Musty--Schiavone--Voight \cite[Example 5.27]{KMSV}: using the algorithm in Lemma \ref{lemma:alg_for_isoms0} we find that $X$ is not   isomorphic to either source curve.  Alternatively, these two curves are minimally defined over $\Q(\sqrt{-7})$ (and are conjugate under $\Gal(\Q(\sqrt{-7})\,|\,\Q)$), whereas $X$ can be defined over $\Q$.  

In the third case, we apply the same argument, appealing to the exhaustive computation of \Belyi\ maps of small degree by Musty--Schiavone--Voight \cite{MSVpreprint} and again checking for isomorphism.
\end{example}

\bibliography{refsbel}{}
\bibliographystyle{plain}

\end{document}